\documentclass[11pt,leqno]{article}
\usepackage{amsmath, amsfonts, theorem,color,
}
\usepackage[latin1]{inputenc}
\usepackage[T1]{fontenc}

\usepackage[english]{babel}

\usepackage{lmodern}

\usepackage{amsmath}

\usepackage{amssymb}

\usepackage{appendix}

\usepackage{mathrsfs}
\usepackage{latexsym}
\usepackage{graphicx}
\usepackage{cancel}
\usepackage[normalem]{ulem}
\makeatletter
\def\@maketitle{\newpage
    \null
    \vskip .8truein
    \begin{center}%
     {\bf \@title \par}%
     \vskip 1.5em
     {\small
      \lineskip .5em
      \begin{tabular}[t]{c}\@author
      \end{tabular}\par}%
    \end{center}%
    \par
    \vskip .4truein}
\@addtoreset{equation}{section} \@addtoreset{theorem}{section}
\@addtoreset{lemma}{section} \@addtoreset{proposition}{section}
\@addtoreset{definition}{section} \@addtoreset{corollary}{section}
\@addtoreset{remark}{section}

\def\dfrac#1#2{\ds{\frac{#1}{#2}}}

\newcommand{\re}{{\mathbb R}}

\newcommand{\Z}{{\mathbb Z}}
\newcommand{\He}{{\mathbb H}}
\newcommand{\cH}{{\mathcal H}}
\newcommand{\Te}{{\mathbb T}_{\cH}}

\let\ds=\displaystyle

\def\R{{\mathbb R}}
\def\N{{\mathbb  N}}
\def\N{{\mathbb N}}

\newtheorem{theorem}{Theorem}[section]
\newtheorem{lemma}{Lemma}[section]
\newtheorem{proposition}{Proposition}[section]
\newtheorem{definition}{Definition}[section]

\newtheorem{remark}{Remark}[section]
\newtheorem{example}{Example}[section]
\newtheorem{hypothesis}{Hypothesis}[section]

\newenvironment{Proofc}[1]{\smallskip\par\noindent\textsc{#1}\quad}%
  {\hfill$\Box$\bigskip\par}

\DeclareMathOperator{\diver}{div}

\def\proof{\list{}{\setlength{\leftmargin}{0pt}
                      \parskip=0pt\parsep=0pt\listparindent=2em
                      \itemindent=0pt}\item[]\futurelet\testchar\@maybe}

\def\@maybe{\ifx[\testchar \let\next\@Opt
          \else \let\next\@NoOpt \fi \next}
\def\@Opt[#1]{{\it Proof of #1.\ }}\def\@NoOpt{{\it Proof.\ }}
\begin{document}
\title{\Large \bf The continuity equation in the Heisenberg-periodic case: a representation formula and an application to Mean Field Games.}

\author{{\large \sc Alessandra Cutr\`\i\thanks{Dipartimento di Matematica, Universit\`a di Roma Tor Vergata, cutri@mat.uniroma2.it}, Paola Mannucci\thanks{Dipartimento di Matematica ``Tullio Levi-Civita'', Universit\`a di Padova, mannucci@math.unipd.it}, Claudio Marchi \thanks{Dipartimento di Matematica ``Tullio Levi-Civita'', Universit\`a di Padova, claudio.marchi@unipd.it}, Nicoletta Tchou \thanks{Univ Rennes, CNRS, IRMAR - UMR 6625, F-35000 Rennes, France, nicoletta.tchou@univ-rennes1.fr}}} 
\maketitle

%
%\noindent {\bf Keywords}: Mean Field Games, first order Hamilton-Jacobi equations, continuity equation, Fokker-Planck equation, noncoercive Hamiltonian, Heisenberg group, degenerate optimal control problem.
%
%\noindent  {\bf 2010 AMS Subject classification:} 35F50, 35Q91, 49K20, 49L25....
%

\begin{abstract}
We provide a representation of the weak solution of the continuity equation on the Heisenberg group $\He^1$ with periodic data (the periodicity is suitably adapted to the group law). This solution is the push forward of a measure concentrated on the flux associated with the drift of the continuity equation. Furthermore, we shall use this interpretation for proving that weak solutions to first order Mean Field Games on $\He^1$ are also mild solutions.
\end{abstract}

\section{Introduction}
This paper is devoted to provide a representation of the weak periodic solution of the continuity equation in the Heisenberg group~$\He^1$
\begin{equation}\label{cont}
\left\{\begin{array}{lll}
\partial_t m(x,t)-\diver_{\cH}  (v(x,t)\, m(x,t))=0&\qquad \textrm{in }\He^1\times (0,T)\\
m(x,0)=m_0(x)&\qquad \textrm{on }\He^1,
\end{array}\right.
\end{equation}
where $m(x,t)$ is the density of a Borel family of measures $m(t)$, the drift $v$ is a Borel vector field and $\diver_{\cH}$ is the {\it horizontal divergence} given by the vector fields generating~$\He^1$ (see Section~\ref{Preliminaries}). 
This solution is the push forward of a measure concentrated on the flux associated with the drift of the continuity equation. 
For this reason it will be called ``probabilistic representation''.
Note that, in Euclidean coordinates, the differential equation in~\eqref{cont} becomes
\begin{equation*}
\partial_t m(x,t)-\diver (v(x,t)\, m(x,t)\, B^T(x))=0 \qquad \textrm{in }\re^3\times (0,T)
\end{equation*}
where, for $x=(x_1, x_2,x_3)\in \re^3$, 
\begin{equation}\label{matrixB}
B(x):=  \begin{pmatrix}
\!\!&1& 0\!\\
\!\!&0&1&\!\\
\!\!&-x_2&x_1\!
\end{pmatrix}\in M^{3\times 2}
\end{equation}
is the matrix associated with the vector fields generating $\He^1$.
We suppose that  the drift $v$ is bounded and $Q_\cH$-periodic in the sense of section~\ref{sub:periodicity} and we study $Q_\cH$-periodic solutions~$m$. More precisely, 
in our main Theorem~\ref{821} we get a representation formula analogous to the one in ~\cite[Theorem 8.2.1]{AGS}. Indeed we note that, if the periodic problem is interpreted as a problem in $\He^1$, then ~\cite[Theorem 8.2.1]{AGS} does not apply because the global summability assumption   \cite[equation (8.1.21)]{AGS} for the drift does not hold.  Nevertheless in our previous paper  ~\cite{MMT} we obtained a representation formula as in ~\cite[Theorem 8.2.1]{AGS} in the Heisenberg group but we required the compactness of the support of $m_0$.
 As for the classical case in \cite{AGS}, to get the probabilistic representation of the solution to~\eqref{cont},
the key ingredient is a ``superposition principle'' in the Heisenberg periodic setting (see \eqref{superA}) which allows to prove that there exists a measure concentrated on the solutions of the characteristic system of ODE 
\begin{equation}\label{ODE}
{Y}'(t)=v(Y(t),t)\,B^T(Y(t))\quad \textrm{for }t\in [0,T],\qquad 
Y(s)=x, \quad s\in [0,T].
\end{equation}
To this end, the key results are Lemma~\ref{8110nostro} and Lemma~\ref{nostro819} which rely on the properties of the distance on~$\He^1$ and on the pavage in~$\He^1$. 

As an application of this result we study evolutive $Q_\cH$-periodic first order Mean Field Games (briefly, MFG) in the Heisenberg group~$\He^1$:
\begin{equation}\label{eq:MFGintrin}
\left\{\begin{array}{lll}
(i)&\quad-\partial_t u+\frac{|D_{\cH}u|^2}{2}=F[m(t)](x)&\qquad \textrm{in }\He^1\times (0,T)\\
(ii)&\quad\partial_t m-\diver_{\cH}  (m D_{\cH}u)=0&\qquad \textrm{in }\He^1\times (0,T)\\
(iii)&\quad m(x,0)=m_0(x),\quad u(x,T)=G[m(T)](x)&\qquad \textrm{on }\He^1,
\end{array}\right.
\end{equation}
where $D_{\cH}$ is the {\it horizontal gradient}, the couplings $F$ and $G$ are strongly regularizing operators and $m_0$ is the initial periodic distribution of players.
Here, the Hamiltonian is noncoercive in the gradient term.
Let us recall that MFG have been introduced by~\cite{LL1,LL2} for describing the interaction of an infinite population of rational and indistinguishable agents.
In the MFG systems, the functions~$u$ and~$m$ are respectively the value function for the generic player and the density of the population; this interpretation motivates the fact that the data are $u(x,T)$ and $m(x,0)$, namely that the MFG systems are of forward-backward type. In this model the agents have forbidden directions: they follow ``horizontal'' trajectories given in terms of the vector fields generating the Heisenberg group.
The Heisenberg group can be seen as a non-Euclidean space which is endowed with a (noncommutative) group operation, a family of dilations and a sub-Riemannian structure. This framework guarantees that any couple of points in $\He^1$ can
be connected by a concatenation of a finite number of ``horizontal'' trajectories (by Chow's theorem). We remark that, differently from our previous results in \cite{MMMT, MMT}, we obtain the existence of a weak solution to problem~\eqref{eq:MFGintrin} by a vanishing viscosity method with the horizontal Laplacian instead of the Euclidean one; this procedure is needed for preserving the $Q_\cH$-periodicity of the problem.
Afterwards, using Theorem~\ref{821}, we prove that this solution is in fact a mild solution in the sense introduced in~\cite{CC}.\\

\vskip .2cm
\noindent\underline{Notations.}
For any function $u:\re^3\times\re\ni (x,t)\to u(x,t)\in \re$, $Du$ and~$D^2u$ stand for the Euclidean gradient and respectively Hessian matrix with respect to~$x$.
For any compact set $A$ of $\re^3$, we denote by $C^2(A)$ the space of functions with continuous second order derivatives endowed with the norm $\|f\|_{C^2(A)}:=\sup_{x\in A}[|f(x)|+|Df(x)|+|D^2f(x)|]$.
For any open set $A$ of $\re^3$, we denote by $L^\infty(A)$ the space of functions~$f:A\to \re$ with $\textrm{ess}\sup f<\infty$.\\
For any complete separable metric space $X$, ${\mathcal M}(X)$ and ${\mathcal P}(X)$ denotes the set of nonnegative Radon measures on~$X$, and respectively of Borel probability measures on~$X$. For any complete separable metric spaces $X_1$ and $X_2$, any measure $\eta\in{\mathcal P}(X_1)$ and any function $\phi:X_1\to X_2$, we denote by $\phi\#\eta\in{\mathcal P}(X_2)$ the {\it push-forward} of~$\eta$ through~$\phi$, i.e.  $\phi\#\eta(B):=\eta(\phi^{-1}(B))$, for any $B$ measurable set, $B\subset X_2$ (see~\cite[section~5.2]{AGS}). For a function $m\in C^0([0,T], {\mathcal P}(X))$, $m_t$ stands for the probability $m(\cdot, t)$ on $X$.

%%%%%%%%%%%%%%%%%

\section{Preliminaries: the Heisenberg group $\He^1$}
\label{Preliminaries}
We introduce the following noncommutative group structure on $\re^3$. We refer to \cite{BLU} for a complete overview on the Heisenberg group. 

\begin{definition}
\label{Hei}
The $3$-dimensional Heisenberg group $\He^1$ is the space $\R^3$,
endowed with the following noncommutative group operation: $\forall x=({x}_1,{x}_2,x_3)$, $y=({y}_1,{y}_2,y_3)\in \R^{3}$, 
\begin{equation}
\label{Group_Law}
x\oplus y:= (x_1+y_1, x_2+y_2, x_3+y_3-x_2y_1+x_1y_2).
\end{equation}

\end{definition}
The law  $x\oplus y$ is called the $x$ {\em {left translation of}} $y$.
We call  $x^{-1}$ the point such that $x^{-1}\oplus x=x\oplus x^{-1}=0$. Note that $x^{-1}=(-x_1, -x_2, -x_3)$. 
In $\He^1$ we define a dilations' family as follows. 
\begin{definition}
\label{DilH}
The dilations in the Heisenberg group are the family of group homeomorphisms defined as, for all $\lambda>0$,  
$\delta_{\lambda}:\He^1\to \He^1$  with
\begin{equation}
\label{Dialtions}
\delta_{\lambda}(x)=(\lambda\, {x}_1,\lambda\, {x}_2, \lambda^2\,x_3),
\quad \forall\;
x=({x}_1,  {x}_2, x_3)\in \He^1.
\end{equation}
\end{definition}
We say that $\He^1$ is generated by the two vector fields associated with the columns of~$B$,
\begin{equation}\label{vectorFields}
 X_1(x)=\left(\begin{array}{c}1 \\0 \\
 -x_2\end{array}\right)\quad
\textrm{and}
\quad 
X_2(x)=\left(\begin{array}{c}0 \\1 \\x_1\end{array}\right),
\quad \forall\, x=(x_1,x_2,x_3)\in \He^1.
\end{equation}
By these vectors we define the linear differential operators, still called $X_1$ and $X_2$
\begin{equation}\label{VFD}
X_1=\partial_{x_1}-x_2\partial_{x_3},\ X_2=\partial_{x_2}+x_1\partial_{x_3}.
\end{equation}
Note that their commutator $[X_1,X_2]:=X_1\,X_2-X_2X_1$ verifies: $[X_1,X_2]=2\partial_{x_3}$; hence, together with their commutator $[X_1,X_2]$, they span all $\re^3$. 
The fields $X_1$ and $X_2$ are left-invariant vector fields, i.e. 
for all $u\in C^{\infty}(\He^1)$ and for all fixed $y\in \He^1$
\begin{equation}\label{LIVF}
X_i(u(y\oplus x))=(X_iu)\,(y\oplus x),\ i=1,2.
\end{equation}
For any regular real-valued function $u$, we shall denote its horizontal gradient and its horizontal Laplacian by $D_\cH u:= (X_1u, X_2u)$ and respectively $\Delta_\cH u:=X_1^2u+X_2^2u$ and we observe $D_\cH u=Du B$ and $\Delta_\cH u=\textrm{tr}(D^2u\,BB^T)$.
For any regular $u=(u_1,u_2):\He^1\to \re^2$, we denote its horizontal divergence by $\diver_{\cH}u:=X_1u_1+X_2u_2$ and we note that the left-invariance of $X_i$ ($i=1,2$) entails the left-invariance of $\diver_{\cH}$. We have: $\diver_{\cH}(D_\cH u)=\Delta_\cH u$.\\
The norm and the distance associated with the group law are defined as:
\begin{equation}\label{norm}
\|x\|_\cH:=((x_1^2+x_2^2)^2+x_3^2)^{1/4},\quad\quad   d_\cH(x,y):=\|y^{-1}\oplus x\|_\cH.
\end{equation}

For any domain $U\subset \He^1\times[0,T]$, any $k\in\N$ and any $\delta\in(0,1]$, we denote by $C^{k+\delta}_\cH(U)$ (resp. $C^{k+\delta}_{\cH, loc}(U)$) the (resp. local) parabolic H\"older space adapted to the vector fields~$X_1$ and~$X_2$ (for instance, see~\cite[Section 4]{BB07} or~\cite[Definition 10.4]{BBLU}). For $\delta=0$ or $k=0$, we simply denote $C^{k}_\cH(U)$ and respectively $C^{\delta}_\cH(U)$.

\subsection{Periodicity in the Heisenberg group}\label{sub:periodicity}

The notion of periodicity is introduced by the law $\oplus$.
Let $Q_\cH =[0,1)^3$.
We can construct a tiling of $\He^1$ by the {\em property of pavage}:
for every $x\in\He^1$ there exists a unique $n_\cH(x)\in \Z^3$, such that there exists a unique $q_\cH =q_\cH (x)\in Q_\cH $ with $n_\cH(x)\oplus q_\cH =x$.
Following~\cite{BMT2,BMT} (see also  \cite{BW}), we define the {\em $Q_\cH $-periodicity} on $\He^1$ with respect to this reference pavage.
\begin{definition}\label{QHper}
A function $f$ on $\He^1$ is said $Q_\cH $-periodic if  
$$f(x)=f(q_\cH (x))\qquad\forall x\in\He^1.$$
\end{definition}
We will denote by $C^{\infty}_{Q_\cH , per}$ the set of the functions $f\in C^{\infty}(\He^1)$ that are $Q_\cH $-periodic.
The definition of $Q_\cH $-periodicity is equivalent to the following definition of {\em $1_\cH $-periodicity}:
\begin{definition}
A function $f$ defined on $\He^1$ is said $1_\cH $-periodic if there holds
$$f(n\oplus x)=f(x) \qquad  \forall x\in\He^1,\, n\in \Z^3.$$
\end{definition}
\begin{lemma}\label{lemma:period}
A function $f$ is $Q_\cH $-periodic if and only if it is $1_\cH $-periodic.
\end{lemma}
\begin{proof}
By the pavage property if $f$ is $1_\cH $-periodic then is $Q_\cH $-periodic. Conversely, for any $x\in\He^1$ there exist unique $n_\cH(x) $ and $q_\cH(x) $ such that $x=n_\cH\oplus q_\cH $. For any $n'\in\Z^3$ we write $n'\oplus x= n'\oplus n_\cH(x)\oplus q_\cH(x)$. Since $n'\oplus n_\cH\in\Z^3$ then $q_\cH (n'\oplus x)= q_\cH (x)$ and by the $Q_\cH $-periodicity we get $f(n'\oplus x)=f(q_\cH (n'\oplus x))= f(q_\cH (x))=f(x)$, for any $n'\in\Z^3$.
\end{proof}
\begin{definition}\label{Htoro}
We denote by $\Te$ the torus in the Heisenberg group~$\He^1$, namely $\He^1/\Z^3$ using the following equivalence law: $x\sim y$ if there exists $n\in \Z^3$ such that $n\oplus x=y$.
The torus is naturally endowed with the distance induced by $d_\cH$: for any $x,y\in \Te$,
\begin{equation*}
d_{\Te}(x,y):=\inf d_\cH(x',y')
\end{equation*}
where the infimum is performed over all the couples $(x',y')\in \He^1\times \He^1$ with $x\sim x'$, $y\sim y'$.
\end{definition}
\begin{remark} Lemma~\ref{lemma:period} ensures that $x\sim x'$ if and only of $q_\cH (x)=q_\cH (x')$. It is worth to observe that the Heisenberg torus~$\Te$ does not coincide with the Euclidean torus; especially, $\Te$ is not obtained identifying the points of two opposite faces of~$\overline{Q_\cH}$ with the same two coordinates. As a matter of facts, this happens between the two faces given by~$x_3=0$ and~$x_3=1$. For completeness, let us write the identification of points $(1,x_2,x_3)$ with $(x_2,x_3)\in [0,1)^2$ with points $(0,x_2',x_3')$ with $(x_2',x_3')\in [0,1)^2$: we have
\begin{equation*}
(1,x_2,x_3)\sim\left\{\begin{array}{ll}
(0,x_2,x_3-x_2) &\quad \textrm{for } x_3-x_2\in [0,1)\\
(0,x_2,x_3-x_2+1) &\quad \textrm{for } x_3-x_2\in [-1,0);
\end{array}\right.
\end{equation*}
actually, for $x_3-x_2\in [0,1)$ there holds $(-1,0,0)\oplus (1,x_2,x_3)= (0,x_2,x_3-x_2)$ while for $x_3-x_2\in [-1,0)$ there holds $(-1,0,1)\oplus (1,x_2,x_3)= (0,x_2,x_3-x_2+1)$. Moreover, $(1,1,x_3)\sim (0,0,x_3)$ because $(-1,-1,0)\oplus (1,1,x_3)=(0,0,x_3)$ for every $x_3\in[0,1)$; $(1,x_2,1)\sim (0,x_2,1-x_2)$ because $(-1,0,0)\oplus (1,x_2,1)=(0,x_2,1-x_2)$ for every $x_2\in[0,1)$ and $(1,1,1)\sim(0,0,0)$ because $(-1,-1,-1)\oplus(1,1,1)=(0,0,0)$. Similarly for the remaining cases.
\end{remark}
\begin{remark}\label{rmk:misureper}
With a slight abuse of notations, throughout this paper we shall identify any measure $\eta\in{\mathcal M}(Q_\cH)$ with the same measure on the torus~$\Te$ and also with the measure $\eta'\in{\mathcal M}(\He^1)$ such that $\eta'(n\oplus A)=\eta (A)$ for any measurable set $A\subset Q_\cH$ and $n\in\Z^3$.
\end{remark}

\subsection{Convolution on Heisenberg group}\label{sec:conv}
We define the convolution of a function $\psi\in L^1_{loc}(\He^1)$ by a function $\rho\in C^{\infty}_c(\He^1)$ as
\begin{equation}\label{conH}
(\psi \ast\rho)(x)= \int_{\He^1} \psi(y)\rho(x\oplus y^{-1}) dy=  \int_{\He^1} \psi(y^{-1}\oplus x)\rho(y) dy.
\end{equation}
We will use the convolution by the regularizing kernel 
\begin{equation}\label{rhoeps}
\rho_{\varepsilon}(x)= C(\varepsilon) \rho_0(\|\delta_{1/\varepsilon}(x)\|^4_\cH )
\end{equation}
where $\rho_0(t)=e^{-t}$ and the constant $C(\varepsilon)$ is chosen such that 
$\int_{\He^1}\rho_{\varepsilon}(x)dx=1$. 
For this convolution the following proposition holds true.
\begin{proposition}\label{propconvH}
We have
\begin{enumerate}
\item [(i)] $\psi \ast\rho_{\varepsilon}=\rho_{\varepsilon}\ast\psi$;
\item [(ii)] If $\psi$ is $Q_\cH$-periodic then also $\psi \ast\rho_{\varepsilon}$ is $Q_\cH$-periodic;
\item [(iii)] If $\psi$ is $L^p(\He^1)$ for some $p\geq 1$, then  $\psi \ast\rho_{\varepsilon}$ is $C^{\infty}(\He^1)$;
\item [(iv)] If $\psi$ is $L^1_{loc}(\He^1)$ then $\psi\ast\rho_{\varepsilon}\to \psi$ in $L^1_{loc}(\He^1)$ as $\varepsilon\to 0$;
\item [(v)]  If $\psi$ is differentiable then $X_i\psi \ast\rho_{\epsilon}=(X_i\psi) \ast\rho_{\epsilon}$; 
\item[(vi)] If $\psi\geq 0$ in $\He^1$ and $\int_{\He^1}\psi(x)dx=C>0$ then $\psi \ast\rho_{\varepsilon}(x)>0$ for any $x\in\He^1$.
\end{enumerate}
\end{proposition}
The proof is standard and relies on the fact that the Haar measure associated with~$\He^1$ coincides with the Lebesgue one (see \cite[Proposition 1.3.21]{BLU}); hence, we shall omit it. We only note that, to prove $(v)$, we use the left invariance of the vector fields $X_i$.

\section{A probabilistic representation for the continuity equation}\label{probsect}

This section contains the main result of this paper, a probabilistic representation of the $Q_\cH$-periodic solution of equation~\eqref{cont}. We adapt the techniques introduced in~\cite[Theorem 8.2.1]{AGS} finding a measure concentrated on the solutions of~\eqref{ODE}. Note that~\cite[Theorem 8.2.1]{AGS} does not apply directly to our case because the global summability assumption \cite[equation (8.1.21)]{AGS} does not hold for the drift $vB^T$.\\
We need some assumptions and definitions. We assume that the set~${\mathcal P}(\Te)$ is endowed with the Kantorovich-Rubinstein distance~${\bf d_{1}}$ (see \cite{BK}):
$$
{\bf d_{1}}(m, m'):= \inf_{\pi\in\Pi(m,m')}\int_{\Te\times \Te}d_{\Te}(x,y)d\pi(x,y)\qquad \forall m,m'\in {\mathcal P}(\Te)
$$
where 
\begin{equation}\label{Pi}
\Pi(m, m'):=\{\pi\in{\mathcal P}(\Te\times \Te): \pi(A\times \Te)=m(A), \pi(\Te \times A)=m'(A)\},
\end{equation}
where $A$ is any Borel set $A\subset \Te$.
We set
\begin{equation}\label{eq:Pper}
{\mathcal P}_{per}(\He^1):=\left\{m\in {\mathcal M}(\He^1):\ m_{\mid Q_\cH}\in {\mathcal P} (Q_\cH),\quad\textrm{$m$ is $Q_\cH$-periodic}\right\}
\end{equation}
where for ``$m$ is $Q_\cH$-periodic'' we mean $m(n\oplus A)=m(A)$ for every $n\in \Z^3$ and every measurable $A\subset Q_\cH$.
By Remark~\ref{rmk:misureper}, we identify ${\mathcal P}_{per}(\He^1)$ with ${\mathcal P}(\Te)$. In particular, by this identification, we assume that also ${\mathcal P}_{per}(\He^1)$  is endowed with the Kantorovich-Rubinstein distance~${\bf d_{1}}$.\\
Let $\Gamma:= AC([0,T], \He^1)$. For each $t\in[0,T]$, the {\it evaluation map} is the map $e_t:\Te\times \Gamma\rightarrow \Te$ with $e_t(x,\gamma)=\gamma(t)$.\\
We can now state our assumptions and our main result whose proof is postponed at the end of this Section.\\
\noindent {\bf {Assumptions $(H)$}}\\
$m_{0} \in$ $\mathcal{P}_{per}(\He^1)$; let $v:\He^1\times[0,T]\to \re^2$, with $v=v(x,t)$, be measurable, bounded, $Q_\cH$-periodic with respect to~$x$ and $v(\cdot, t)$ is Borel for every $t\in[0,T]$.
\begin{theorem}\label{821}
Let $m:[0, T] \rightarrow \mathcal{P}(\Te)$ be a narrowly continuous solution of problem~\eqref{cont}. Under Assumptions $(H)$, there exists a probability measure $\eta$ in $\Te\times \Gamma$, 
 such that\\
(i) $\eta$ is concentrated on the set of pairs $(x, \gamma)$ such that $\gamma \in \Gamma$ solves~\eqref{ODE} with $s=0$.\\
(ii) $m_t=m^{\eta}_t:=e_{t}{\#} \eta$ for any $t \in[0, T]$, namely:
\begin{equation}\label{821nostra}
\int_{\Te} \varphi d m^{\eta}_t:=\int_{\Te \times \Gamma} \varphi(\gamma(t)) d \eta(x, \gamma) \quad \forall \varphi \in C^{0}(\Te), t \in[0, T].
\end{equation}
Conversely, any $\eta$ satisfying (i) induces via \eqref{821nostra} a solution of~\eqref{cont} with $m_{0}=e_0 \# \eta$.
\end{theorem}

We recall that a ($Q_\cH$-periodic) function $m$ is a distributional solution of \eqref{cont} if
\begin{equation}\label{813nostra}
\int_0^T\int_{\mathbb{H}^{1}}(\partial_t\varphi- v\,\cdot D_{\cH}\varphi)m(x,t)dx dt=0\qquad \forall \varphi\in C_{c}^{\infty}(\mathbb{H}^{1}\times (0,T)).
\end{equation}
Choosing $\varphi(t,x)=\eta(t)\zeta(x)$ with $\eta\in C_c^{\infty}(0,T)$ and $\zeta\in C_{c}^{\infty}(\mathbb{H}^{1})$, by density, we get the following equivalent formulation of  \eqref{813nostra}: in the sense of distribution in $(0,T)$ there holds
\begin{equation}\label{814nostra}
\frac{d}{dt} \int_{\mathbb{H}^{1}}\zeta(x)m(x,t)dx=-\int_{\mathbb{H}^{1}}v\cdot D_{\cH}\zeta(x)m(x,t)dx.
\end{equation}
Note that, by periodicity, $m$ is a solution of \eqref{cont} in the sense of distributions in $(0,T)$ also over $\Te$, i.e. \eqref{814nostra} holds also over $\Te$:
\begin{equation}\label{814nostraQH}
\frac{d}{dt} \int_{\Te}\zeta(x)m(x,t)dx=-\int_{\Te}v\cdot D_{\cH}\zeta(x)m(x,t)dx,\qquad  \forall \zeta\in C^{\infty}(\Te).
\end{equation}
%%%%%%%%%%%%%%%%%%%%%%
The following Lemma ensures that any $Q_\cH$-periodic distributional solution to~\eqref{cont} (i.e. satisfying \eqref{814nostraQH}) has a representative in $C([0,T],{\mathcal P}_{per}(\He^1))$ which will be still called $m$.
\begin{lemma}\label{8.1.2}
Let $m:\He^1\times[0,T]\to \re$, with $m=m(x,t)$, be a measurable function satisfying \eqref{814nostraQH} such that, for any $t\in[0,T]$, $m(\cdot, t)$ is the density of a Borel $Q_\cH$-periodic probability measure. 
Then there exists a narrowly continuous curve $t \in[0, T] \mapsto \tilde{m}(x,t) \in \mathcal{P}\left(\Te\right)$ such that $m(\cdot, t)=\tilde{m}(\cdot, t)$ for a.e. $t \in(0, T).$
\end{lemma}
\begin{proof}
We follow the proof of \cite[lemma 8.1.2]{AGS} replacing $D\zeta$ with $D_{\cH}\zeta$, where $\zeta\in C^{\infty}(\Te)$; since $m$ is a measure on~$\Te$, $v$ is bounded and $\Te$ is compact, we get the tightness of the family $m$.
\end{proof}
Now we want to obtain an explicit solution of \eqref{cont} under an additional regularity assumption for~$v$. 
More precisely, let $\operatorname{Lip}\left(\overline v, K\right)$ be the Lipschitz constant w.r.t. $x$ of $\overline v(x,t)$ in the set $K$.
When the drift $ v$ in equation \eqref{cont} satisfies 
\begin{equation}\label{818nostra}
\int_{0}^{T}\operatorname{Lip}\left( v, K\right)d t<\infty\qquad\forall K\subset \He^1,\, \textrm{compact},
\end{equation}
we can obtain an explicit solution of \eqref{cont} by the classical method of characteristics (see Proposition~\ref{818} below). We approximate $v$ and $m$ with $v^{\varepsilon}$ and $m^{\varepsilon}$ by means of the convolution with a family of mollifiers as in~\eqref{conH}-\eqref{rhoeps}. The~$v^{\varepsilon}$ satisfy~\eqref{818nostra} and $m^{\varepsilon}$ solves~\eqref{cont} with drift $v^{\varepsilon}$. Hence, we get a representation formula for $m^{\varepsilon}$. 
Lemma \ref{nostro819} is the key result to get this representation formula.
To prove it we strongly use the properties of the distance associated with the Heisenberg group and of the pavage representing $\mathbb{H}^{1}$.
\noindent To prove Lemma \ref{nostro819} we need this technical Lemma.
%%%%%%%%%%%%%%%%%%%%%%;
\begin{lemma}\label{8110nostro}
Let $m \in \mathcal{P}(\Te)$, $E\in L^{\infty}(\Te)$ absolutely continuous with respect to $m$, $\rho\in C^{\infty}_c(\He^1)$ strictly positive. 
Then, for any $p \geq 1$
$$
\int_{\Te}\left|\frac{E * \rho}{m * \rho}\right|^{p} m * \rho\, d x \leq \int_{\Te}\left|\frac{E}{m}\right|^{p} d m.
$$
\end{lemma}
\begin{proof}
Arguing as in the proof of \cite[Lemma 8.1.10]{AGS}, using the Jensen inequality, for any $x\in \He^1$ we get
\begin{eqnarray*}
\left|\frac{E \ast\rho(x)}{m \ast \rho(x)}\right|^{p} m \ast \rho(x) 
&\leq&\int_{\mathbb{H}^{1}}\left|\frac{E}{m}\right|^{p}(y) \rho(x\oplus y^{-1}) d m(y)\\
&=&\sum_{n\in\Z^3}\int_{n\oplus \Te}\left|\frac{E}{m}\right|^{p}(y) \rho(x\oplus y^{-1})dm(y)\\
&=&\sum_{n\in\Z^3}\int_{\Te}\left|\frac{E}{m}\right|^{p}(n\oplus z) \rho(x\oplus (n\oplus z)^{-1})dm(n\oplus z)\\
&=&\int_{\Te}\left|\frac{E}{m}\right|^{p}(z) \sum_{n\in\Z^3}\rho(x\oplus (n\oplus z)^{-1})dm(z)
\end{eqnarray*}
where we used that $y= n\oplus z$ with $n\in \Z^3$, the $\Te-$periodicity of $m$ and of ${E}/{m}$.
Integrating with respect to $x$ in $\Te$ we get
\begin{eqnarray*}
&&\int_{\Te}\left|\frac{E \ast\rho(x)}{m \ast\rho(x)}\right|^{p} m\ast \rho(x) dx
\leq 
\int_{\Te}\int_{\Te}\left|\frac{E}{m}\right|^{p}(z) \sum_{n\in\Z^3}\rho(x\oplus (n\oplus z)^{-1})dm(z)\,dx\\
&&=\int_{\Te}\left|\frac{E}{m}\right|^{p}(z)\bigg(\sum_{n\in\Z^3}\int_{\Te}\rho(x\oplus (n\oplus z)^{-1})dx\bigg) dm(z)= \int_{\Te}\left|\frac{E}{m}\right|^{p}(z)dm(z).
\end{eqnarray*}
The last equality comes from  
$$\sum_{n\in\Z^3}\int_{\Te}\rho(x\oplus (n\oplus z)^{-1})dx=\int_{\He^1}\rho(y)dy=1$$
and this equality
is due to the fact that, fixed $z\in \Te$,
$$\mathbb{H}^{1}= \cup_{n\in\Z^3} \Te\oplus (n\oplus z)^{-1}.$$
To prove it we have to show that for any $y\in \mathbb{H}^{1}$ there exists an unique $n\in\Z^3$ such that there exists $x\in \Te$ such that $y= x\oplus(n\oplus z)^{-1}$ or equivalently 
$y\oplus (n\oplus z)=x$.
By writing explicitly this last relation, we obtain $x_i=y_i+n_i+z_i$ with $i=1,2$ and 
$x_3=(y\oplus z)_3+n_3-n_2(z_1-y_1)+n_1(z_2-y_2)$, where we denoted by $(y\oplus z)_3$ the third component of $(y\oplus z)$.
Hence we take $n_i= -[y_i+z_i]$
and $x_i= M(y_i+z_i)$,
 $i=1,2$, where $[\cdot]$ and $M(\cdot)$ are respectively the integer part and the fractional part of a real number.
 Analogously 
$n_3=-[(y\oplus z)_3-n_2(z_1-y_1)+n_1(z_2-y_2)]$ and $x_3= M((y\oplus z)_3-n_2(z_1-y_1)+n_1(z_2-y_2))$.
\end{proof}

%%%%%%%%%%%%%%%%%
\begin{lemma}\label{nostro819}
Under Assumption~$(H)$, let $m$ be a time continuous $Q_\cH$-periodic solution of~\eqref{cont}. 
For $\rho_{\varepsilon}$ as in~\eqref{rhoeps}, set
$$
m^{\varepsilon}:=m \ast \rho_{\varepsilon}, \quad E^{\varepsilon}:=\left(v\, m\right) \ast\rho_{\varepsilon}, \quad v^{\varepsilon}:=\frac{E^{\varepsilon}}{m^{\varepsilon}}.
$$
Then $m^{\varepsilon}$, $E^{\varepsilon}$ and $v^{\varepsilon}$ are $Q_\cH$-periodic. Moreover $m^{\varepsilon}$ is a continuous solution of
\begin{equation}\label{8128nostra} 
\partial_{t} m^{\varepsilon}-\diver_{\cH}(v^{\varepsilon}\, m^{\varepsilon})=0,\qquad \textrm{in}\ \mathbb{H}^{1} \times(0, T),
\end{equation}
where $v^{\varepsilon}$ fulfills the regularity property~\eqref{818nostra} and the uniform integrability bound
\begin{equation}\label{8126nostra}
\int_{\Te}\left|v^{\varepsilon}(x,t)\right|^{p} d m_t^{\varepsilon}(x) \leq C, \quad \forall \varepsilon>0,\ \forall t \in(0, T),\ p \geq 1.
\end{equation}
Moreover, as $\varepsilon\to0^+$, $E^{\varepsilon}_t \rightarrow v_t\, m_t$ narrowly and
\begin{equation}\label{8127nostra}
\lim _{\varepsilon \to 0}\left\|v^{\varepsilon}\right\|_{L^p\left(m^{\varepsilon} ; \Te\right)}=\left\|v(\cdot,t)\right\|_{L^p\left(m ; \Te\right)} \qquad \forall t \in(0, T),
\end{equation}
where $\left\|\cdot\right\|_{L^p\left(m ; \Te\right)}$ is the $L^p$ norm w.r.t. $m$ over $\Te$.
\end{lemma}
\begin{proof}
Proposition~\ref{propconvH}-$(ii)$ ensures that $m^{\varepsilon}$, $E^{\varepsilon}$ and $v^{\varepsilon}$ are $Q_\cH$-periodic.
From Proposition~\ref{propconvH}-$(v)$ and the continuity of $m^{\varepsilon}(x,t)$ w.r.t. $x$ and $t$, we get $m^{\varepsilon}(x,t)>0$ for any $(x,t)\in \Te\times[0,T]$.
From the definition of $\rho_ {\varepsilon}$, since $m$ is bounded then $m^{\varepsilon}$ is bounded. From the definition of the ${\cH}$-norm \eqref{norm}
we get that
$$D\rho_ {\varepsilon}(x)=C(\varepsilon)e^{-(\|\delta_{1/\varepsilon}(x)\|^4_\cH)}
\left(
\frac{4x_1(x_1^2+x_2^2)}{\varepsilon^4}, \frac{4x_2(x_1^2+x_2^2)}{\varepsilon^4}, \frac{2x_3}{\varepsilon^4}
\right).
$$
Hence, in $\Te$, the spatial gradient of $m^{\varepsilon}$ is bounded by a constant dependent on $\varepsilon$.
Analogously, in $\Te$, $E^{\varepsilon}$ and its spatial gradient are bounded in space by the product of $\left\|v\right\|_{L^{1}\left(m;\Te\right)}$ with a constant depending on $\varepsilon$. 
Moreover, from the positivity of $m^{\varepsilon}$, the regularity assumptions \eqref{818nostra} for $v^{\varepsilon}$ hold. 
Lemma \ref{8110nostro} shows that \eqref{8126nostra} holds.
From Proposition~\eqref{propconvH}-$(v)$, we get $\diver_{\cH}(v\, m_{t})\ast \rho_{\varepsilon}=\diver_{\cH}(v^{\varepsilon} m^{\varepsilon}_t).$
Since $m$ solves \eqref{cont}, then $m^{\varepsilon}$ solves the continuity equation \eqref{8128nostra}.
Finally,
general lower semicontinuity results on integral functionals defined on measures of the form
$$
(E, m) \mapsto \int_{\Te}\left|\frac{E}{m}\right|^{p} d m
$$
and Lemma \ref{8110nostro} give \eqref{8127nostra}.
\end{proof}

%%%%%%%%%%%%%%%%%%%%%%

In the following lemma we obtain an elementary result on existence and uniqueness for the characteristic system associated with equation~\eqref{8128nostra}.
\begin{lemma}\label{8.1.4}
Let~$v^\varepsilon$ be the field introduced in Lemma~\ref{nostro819}. Then for any $x\in\He^1$ and $s\in[0,T]$, the ODE
\begin{equation}\label{819nostra}
Y'(t)=v^{\varepsilon}\left(Y(t),t\right)\, B^T\left(Y(t)\right),\qquad Y(s)=x
\end{equation}
admits a unique solution which is defined for every ~$t\in [0,T]$.
\end{lemma}
\begin{proof} 
Equation \eqref{819nostra} reads
\[
Y'_{1}= v^{\varepsilon}_{1},\qquad
Y'_{2}= v^{\varepsilon}_{2},\qquad
Y'_{3}= -Y_{2}v^{\varepsilon}_{1}+Y_{1}v^{\varepsilon}_{2}.
\]
The boundedness of $v^\varepsilon$ entails the boundedness of $Y_{1}$ and $Y_{2}$. Afterwards, we deduce the boundedness of~$Y_{3}$. 
Following the procedure of \cite[Lemma 8.1.4]{AGS} we get the result.
\end{proof}

\begin{proposition}\label{818}
Under Assumption~$(H)$, let $m^{\varepsilon}_t\in {\mathcal P}_{per}(\He^1)$, $t \in[0, T]$, be a family of narrowly continuous measures solving equation~\eqref{8128nostra} with $m^{\varepsilon}_0:=m_0\ast\rho_{\varepsilon}$. 
Then for $m^{\varepsilon}_{0}$-a.e. $x \in \He^1$ the characteristic system \eqref{819nostra} with $s=0$ admits a solution $Y^{\varepsilon}$ on $[0, T]$ and
\begin{equation}\label{8120nostra}
m^{\varepsilon}_{t}=Y^{\varepsilon}(t){\#} m^{\varepsilon}_{0}, \quad \forall t \in[0, T].
\end{equation}
\end{proposition}
\begin{proof}
The proof follows the steps of  \cite[Lemma 8.1.6, Proposition 8.1.7]{AGS}, 
we only give the sketch of the proof of these steps.\\
1) From Lemma \ref{nostro819}, $v^{\varepsilon}$ satisfies \eqref{818nostra}; hence by Lemma~\ref{8.1.4}, for any $x\in\mathbb{H}^{1}$, \eqref{819nostra} with $s=0$ admits an unique solution $Y^{\varepsilon}(t)$ defined in $[0, T]$.\\
2) $Y^{\varepsilon}(t){\#} m^\varepsilon_{0}$ is a continuous solution of \eqref{8128nostra} (in the topology of $C([0,T],{\mathcal P}_{per}(\He^1))$). 
Indeed,  still from Lemma \ref{nostro819}, the velocity field $v^{\varepsilon}$ satisfies~\eqref{818nostra} and \eqref{8126nostra}. 
At this point we follow the proof of \cite[Lemma 8.1.6]{AGS} where we replace $\mathbb{R}^{d}$ with $\He^1$ and $D\varphi$ with $D_{\cH}\varphi$ with $\varphi \in C_c^{\infty}\left(\Te\times(0, T)\right)$ noting that $\langle D_{\cH}\varphi, v^{\varepsilon}\rangle= \langle D\varphi, v^{\varepsilon}\,B\rangle$.\\
3) We follow the proof of \cite[Proposition 8.1.7]{AGS} replacing $\mathbb{R}^{d}$ with $\He^1$ and $D$ with $D_{\cH}$; we get that  any solution $m^{\varepsilon}$ of \eqref{8128nostra} can be represented as \eqref{8120nostra}.
\end{proof}
\begin{Proofc}{Proof of Theorem \ref{821}.}
We adapt the arguments of the proof of \cite[Theorem 8.2.1]{AGS} and of \cite[Theorem 4.18]{C}, hence we only sketch the key steps.\\
1) For $m$ as in the statement, we apply Lemma~\ref{nostro819} and we find $Q_\cH$-periodic approximations $m^{\varepsilon}, v^{\varepsilon}$ satisfying the equation~\eqref{8128nostra}.
Therefore, by Proposition~\ref{818}, we obtain the representation formula $m^{\varepsilon}=Y^{\varepsilon}\# m_{0},$ where $Y^{\varepsilon}$ is the solution of~\eqref{819nostra} with $s=0$. \\
2) Since $Y^\varepsilon$ naturally induces a map from~$\Te$ to~$\Gamma$, we define the measure $\eta^\varepsilon\in {\mathcal P}(\Te\times \Gamma)$ as $\eta^\varepsilon:= (i\times Y^\varepsilon)\#m_0$ where $(i\times Y^\varepsilon):\Te\to \Te\times \Gamma$ with $(i\times Y^\varepsilon)(x):=(x,Y^\varepsilon)$ and $Y^\varepsilon $ denotes the solution of~\eqref{819nostra} with $s=0$. In other words, for any Borel function~$\phi$ defined in~$\Te\times\Gamma$, the measure~$\eta^\varepsilon$ verifies
\begin{equation*}
\int_{\Te\times\Gamma}\phi(x,\gamma)d \eta^\varepsilon(x,\gamma)=\int_{\Te}\phi(x,Y^\varepsilon)d m_0(x).
\end{equation*}
Following  the procedure of \cite[Theorem 4.18]{C} we prove that $(\eta^\varepsilon)$ is a relatively compact family of measures on $\Te\times\Gamma$
and that, if $\eta$ is a narrow cluster point of~$(\eta^\varepsilon)$, then $m_t$ can be represented by $m_t=e_t\#\eta$ and $m_0$ is the first marginal of~$\eta$.\\ 
3) To show that $\eta$ is concentrated on the solutions of the differential equation \eqref{ODE}, we follow the arguments of the proof of \cite[Theorem 8.2.1]{AGS} and
we get the following ``superposition principle''
\begin{equation}\label{superA}
\int_{\Te\times \Gamma}\left|\gamma(t)-x-\int_{0}^{t} v(\gamma(\tau), \tau)\, B^T(\gamma(\tau))d \tau\right| d \eta(x, \gamma)=0 \quad \forall t \in[0, T].
\end{equation}
Then we disintegrate $\eta$ with respect to its first marginal $m_0$ (see~\cite[pag 122]{AGS} or \cite[Theorem 8.5]{C}):
\begin{equation}\label{dis}
d \eta(x, \gamma)=d\eta_x(\gamma)\,dm_0(x)
\end{equation}
and from \eqref{superA} we get for $m_0$-a.e. $x\in \Te$, $\eta_x$-a.e. $\gamma$ is a solution of the \eqref{ODE}.\\
4) The converse implication is exactly as in \cite{AGS} replacing, as usual, the Euclidean gradient $D$ with the horizontal gradient $D_{\cH}$.
\end{Proofc}

\section{An application to first order MFGs}\label{MFGS}
As application of Theorem~\ref{821}, we study the Mean Field Games system~\eqref{eq:MFGintrin} where $m_0$, $F$ and~$G$ are $Q_\cH$-periodic w.r.t. $x$. 
We remark that in this setting we cannot apply the results obtained in \cite{MMT} because we do not have the global summability assumption for the drift in~\eqref{eq:MFGintrin}-$(ii)$ and we do not assume that $m_0$ has compact support. Here, in order to get the existence of a weak solution, we shall use a vanishing viscosity method with the horizontal Laplacian so to preserve the $Q_\cH$-periodicity of the problem. For second order Mean Field Games under H\"ormander condition, we refer the reader to \cite{DF, MMM}.

Let us recall that MFG system~\eqref{eq:MFGintrin} arises when the generic player with state~$x$ at time~$t$ wants to choose the control $\alpha=(\alpha_1, \alpha_2)\in L^2([t,T];\R^2)$ so to minimize the cost
\begin{equation}\label{Jgen}
J_t^{m}(\gamma,\alpha):=\int_t^T\left[\frac12 |\alpha(\tau)|^2+F[m_\tau](\gamma(\tau))\right]\,d\tau+G[m_T](\gamma(T))
\end{equation}
where $m=(m_t)_{t\in[0,T]}$ is the evolution of the whole population's distribution while $(\gamma,\alpha)$ is an {\it horizontal curve} with respect to the two vector fields $X_1$ and $X_2$ defined in \eqref{vectorFields}:
\begin{equation}\label{DYNH}
\gamma'(s)=\alpha_1(s)X_1(\gamma(s))+\alpha_2(s)X_2(\gamma(s))= \alpha(s)B^T(\gamma(s))\quad \textrm{a.e. }[t,T], \quad \gamma(t)=x.
\end{equation}

%%%%%%%%%%%%%%%%%%%%%%%%%%%%%%%%%

Throughout this section, unless otherwise explicitly stated, we shall require the following hypotheses:
\begin{enumerate}
\item[(H1)]\label{H1} the functions~$F$ and $G$ are real-valued functions, continuous on ${\mathcal P}_{per}(\He^1)\times\He^1$, moreover, for any fixed $m\in{\mathcal P}_{per}(\He^1)$, $F[m](\cdot)$ and $G[m](\cdot)$ are $Q_\cH $-periodic;
\item[(H2)]\label{H2} the map $m\to F[m](\cdot)$ is Lipschitz continuous from~${\mathcal P}_{per}(\He^1)$ to $C^{2}(\re^3)$; moreover, there exist~$C\in \mathbb R$ and $\delta_0\in(0,1]$ such that
$$\|F[m](\cdot)\|_{C^{2+\delta_0}(\re^3)}, \|G[m](\cdot)\|_{C^{2}(\re^3)}\leq C,\qquad \forall m\in {\mathcal P}_{per}(\He^1);$$
\item[(H3)]\label{H4} the function~$m_0:\He^1\to \re$ is nonnegative, $C^0$, $Q_\cH$-periodic with $\int_{Q_\cH }m_0dx=1$.
\end{enumerate}
\begin{example} Easy examples of $F$ and $G$ are given by the convolution of a regular kernel (as the one defined in \eqref{conH}-\eqref{rhoeps}) with~$m$.
In this case, Proposition~\ref{propconvH} ensures that assumptions~(H1) and~(H2) are satisfied.
\end{example}
We now introduce our definitions of solution of the MFG system~\eqref{eq:MFGintrin} and state the main result concerning its existence.
\begin{definition}\label{defsolmfg}
A couple $(u,m)\in W^{1,\infty}(\He^1\times[0,T])\times C^0([0,T];{\mathcal P}_{per}(\He^1))$, is a solution of system~\eqref{eq:MFGintrin} if:
\begin{itemize}
\item[1)] for each $ t\in [0,T]$, $m_t$ is absolutely continuous w.r.t. the Lebesgue measure. Let $m(\cdot, t)$ denote the density of $m_t$. The function $(x,t)\mapsto m(x,t)$ is bounded;
\item[2)] Equation~\eqref{eq:MFGintrin}-$(i)$ is satisfied by $u$ in the viscosity sense in~$\He^1\times(0,T)$ while equation~\eqref{eq:MFGintrin}-$(ii)$ is satisfied by $m$ in the sense of distributions  in~$\He^1\times(0,T)$.
\end{itemize}
\end{definition}
In order to give a more detailed description of the MFG, it is expedient to use the notion of {\it mild} solution, introduced by~\cite{CC} and reminiscent of the Lagrangian approach (see~\cite{BCS}).\\
For any $(x,t)\in \He^1\times(0,T)$, we define 
\begin{equation*}
\mathcal A(x,t):=\{(\gamma, \alpha)\in AC([t,T]; \He^1)\times L^2([t,T]; \re^2);\, (\gamma, \alpha) {\text { solves }}\eqref{DYNH}\}.
\end{equation*}
Recall that $\Gamma$ and $e_t$ are defined in Section~\ref{probsect}.
Given $m_0\in{\mathcal P}_{per}(\He^1)$,  we define 
$$\mathcal P_{m_0}(\Gamma)=\{\eta\in{\mathcal M}(\Gamma):\ m_0 =e_0\# \eta \textrm{ and }e_t\# \eta\in {\mathcal P}_{per}(\He^1), \quad \forall t\in[0,T] \}.$$
For any $\eta\in \mathcal P_{m_0}(\Gamma)$ and $x\in \He^1$, 
we introduce the cost $J^{\eta}_t(\gamma, \alpha):=J^{m^\eta}_t(\gamma, \alpha)$, where $m^\eta:=(e_t\#\eta)_{t\in[0,T]}$, 
and the set of optimal horizontal arcs starting at~$x$ 
\begin{equation}\label{Gammaeta}
\Gamma^{\eta}[x]:= \{\overline\gamma:\ (\overline\gamma, \overline\alpha) \in \mathcal A(x,0): 
J^{\eta}_t(\overline\gamma, \overline\alpha)= \min_{(\gamma, \alpha) \in \mathcal A(x,0)} J^{\eta}_t(\gamma, \alpha)\}.
\end{equation}
\begin{definition}\label{mfgequil}
A measure $\eta\in \mathcal P_{m_0}(\Gamma)$ is a {\em MFG equilibrium} for $m_0$ if $supp\, \eta\subseteq \bigcup\limits_{x\in\He^1} \Gamma^{\eta}[x].$

\end{definition}
\begin{definition}\label{mild}
A couple $(u,m)\in C^0(\He^1\times [0,T])\times C^0([0,T]; \mathcal P_{per}(\He^1))$ is called {\em {mild solution}} of system~\eqref{eq:MFGintrin} if there exists a MFG equilibrium $\eta$ for $m_0$ such that: $m_t =e_t\# \eta$ for any $t\in[0,T]$ and $u(x,t)= \inf_{(\gamma, \alpha) \in \mathcal A(x,t)} J^{\eta}_t(\gamma, \alpha)$.
\end{definition}
Now we can state the main result of this section.
\begin{theorem}\label{thm:main}
Assume $(H1)$-$(H3)$. Then, \\
$(i)$ system \eqref{eq:MFGintrin} has a solution $(u,m)$,\\
$(ii)$ any solution $(u,m)$ is also a mild solution.
\end{theorem}
The proof of point $(i)$ is standard and we shall provide it only for completeness while the proof of point~$(ii)$ relies on the superposition principle proved in Theorem \ref{821}. The proof of Theorem~\ref{thm:main} is postponed in section~\ref{sect:dim3.1}.
\begin{remark} Differently from \cite{AMMT} and \cite{MMMT}, here we cannot obtain the representation of~$m$ as the push-forward of~$m_0$ by the flow associated with the optimal control problem. This is due to the fact that we cannot prove the uniqueness of the optimal trajectories and then we cannot say that $\Gamma^{\eta}[x]$ is a singleton, or equivalently that the disintegrated measure $\eta_x$ (see \eqref{dis}) coincides with the Dirac measure $\delta_{\overline \gamma_x}$.
\end{remark}

\subsection{The Hamilton-Jacobi equation}\label{OC}
In this section, we tackle the optimal control problem associated with the Hamilton-Jacobi equation~\eqref{eq:MFGintrin}-$(i)$. 
Throughout this section we shall assume the following hypothesis
\begin{hypothesis} \label{BasicAss} $f\in C^{0}([0,T],C^2(\re^3))$ and $g\in C^2(\re^3)$ are $Q_\cH $-periodic w.r.t. $x$ and there exists a constant $C$ such that $\sup_{t\in[0,T]}\|f(\cdot,t)\|_{C^2(\re^3)} + \|g\|_{C^2(\re^3)} \leq C.$
\end{hypothesis}

\begin{definition}\label{def:OCD} We consider the following optimal control problem:
\begin{equation}\label{def:OC}
\text{minimize } J_t(\gamma,\alpha):
=\displaystyle\int_t^T\dfrac12|\alpha(s)|^2+f( \gamma(s),s)\,ds+g(\gamma(T))\quad\textrm{over }(\gamma, \alpha)\in \mathcal A(x,t).
\end{equation}
We say that $\gamma^*$ is an optimal trajectory if there is a control $\alpha^*$ such that $(\gamma^*,\alpha^*)\in  \mathcal A(x,t)$ is optimal for problem~\eqref{def:OC}.
\end{definition}
\begin{remark} \label{2.3} 
For any $(x,t)\in Q_{\cH}\times[0,T)$, we claim that problem~\eqref{def:OC} admits a solution $(\gamma^*,\alpha^*)$ which, moreover, fulfills $\|\alpha^*\|^2_{L^2(t,T)}\leq 2C[(T-t)+1]$ and $\gamma^*\in C^{1/2}([t,T], \He^1)$, where $C$ is the constant introduced in Hypotheses~\ref{BasicAss}. Indeed, for $(x,t)$ fixed, the boundedness of $f$ and $g$ entail that the infimum in~\eqref{def:OC} is bounded from below by $-(\|f\|_\infty T+\|g\|_\infty)$. Moreover using the trajectory $(\bar \gamma(s),\bar \alpha(s))=(x,0)$ for every $s\in[t,T]$, we obtain that this infimum is bounded from above by $\|f\|_\infty T+\|g\|_\infty$. We now consider a minimizing sequence $\{(\gamma_n,\alpha_n)\}_n$. By the last estimate, we get $\|\alpha_n\|^2_{L^2(t,T)}\leq 2(\|f\|_\infty T+\|g\|_\infty)$. Hence, possibly passing to a subsequence (that we still denote $\{(\gamma_n,\alpha_n)\}_n$), we can assume that the sequence~$\{\alpha_n\}_n$ is $L^2(t,T)$-weakly convergent to some function $\alpha\in L^2(t,T)$. By H\"older inequality we infer that $\gamma_n$ belong to $C^{1/2}([t,T], Q_{\cH})$ (see also \cite[Remark 3.1]{MMT} for a similar argument).
\end{remark}
\begin{definition} The value function for the cost $J_t$ defined in \eqref{def:OC} is
\begin{equation}\label{repr}u(x,t):=\inf\left\{ J_t(\gamma, \alpha):\, (\gamma,\alpha)\in \mathcal A(x,t)\right\}.
\end{equation}
An optimal couple $(\gamma^*, \alpha^*)$ for problem~\eqref{def:OC} is also said to be optimal for $u(x,t)$.
\end{definition}
The following lemma permits to restrict our study to $Q_\cH$ because the value function $u$ is $Q_\cH$-periodic in~$x$. %Its proof relies on standard arguments and on the property: if $\gamma_1$ and $\gamma_2$ solves \eqref{DYNH} with the same control $\alpha$ and with $\gamma_1(t)=x$ and $\gamma_2(t)=n\oplus x$ for $n\in\Z^3$, then $\gamma_2(\cdot)= n\oplus \gamma_1(\cdot)$. Hence, we shall omit its proof.
\begin{lemma}\label{uper}
Let $u$ be the value function introduced in \eqref{repr}. Then $u$ is $Q_\cH $-periodic in~$x$.
\end{lemma}
\begin{proof}
Note that if $x(s)$ and $y(s)$ solves \eqref{DYNH} with the same law of control $\alpha$ and with respectively $x(t)=x$ and $y(t)=z\oplus x$, then 
$y(s)= z\oplus x(s)$; actually there hold
\begin{eqnarray*}
y_i(s) &=& z_i+x_i+\int_t^s\alpha_i(\tau)d\tau= z_i+ x_i(s),\qquad \textrm{for } i=1,2,\\
y_3(s) &=& z_3+x_3-z_2x_1+z_1x_2+\int_t^s(z_2+ x_2(\tau))(-\alpha_1(\tau))+ (z_1+ x_1(\tau))\alpha_2(\tau)d\tau\\
&=&z_3+\left(x_3- \int_t^sx_2(\tau)\alpha_1(\tau)- x_1(\tau)\alpha_2(\tau)d\tau\right)-z_2\left(x_1+\int_t^s\alpha_1(\tau)d\tau\right)\\
&&+z_1\left(x_2+\int_t^s\alpha_2(\tau)d\tau\right)\\
&=&z_3 +x_3(s)-z_2 x_1(s)+z_1 x_2(s).
\end{eqnarray*}
Taking advantage of the $Q_\cH$-periodicity of~$f$ and~$g$, for any $z\in\Z^3$, we deduce
\begin{eqnarray*}
u(z\oplus x,t)
%&=&\inf_{\beta}\displaystyle\int_t^T\dfrac12|\beta(s)|^2+f(y(s),s)\,ds+g(y(T))\\
&=&\inf_{\alpha}\displaystyle\int_t^T\dfrac12|\alpha(s)|^2+f(z\oplus x(s),s)\,ds+g(z\oplus x(T))\\
&=&
\inf_{\alpha}\displaystyle\int_t^T\dfrac12|\alpha(s)|^2+f(x(s),s)\,ds+g(x(T))= u(x,t)
\end{eqnarray*}
where the infimum is taken among all the possible controls~$\alpha$; hence, the value function is $Q_\cH$-periodic.
\end{proof}

The following proposition permits to restrict our study on uniformly bounded controls. We shall omit its proof because it follows the same arguments of the proof of \cite[Proposition 3.1]{MMT}, using the fact that $Q_\cH$ is a bounded set.
\begin{proposition}\label{boundalfa}
  Let $u$ be the value function introduced in \eqref{repr}. Then, there exists a constant $C_2$ (depending only on $T$ and on the constant $C$ of Hypothesis \ref{BasicAss}) such that: $u(x,t)=\inf \{J_t(\gamma, \alpha):\ (\gamma, \alpha) \in\mathcal A(x,t),\ \|\alpha\|_{\infty}\leq C_2\}$ for any $(x,t)\in Q_\cH\times[0,T]$.
\end{proposition}
We now study the Hamilton-Jacobi equation associated with the problem of Definition~\ref{def:OCD}: 
\begin{equation}\label{eq:HJ1}
-\partial_t u+\frac{|D_{\mathcal H}u|^2}{2}=f(x, t)\quad \textrm{in }\He^1\times (0,T),\qquad u(x,T)=g(x)\quad \textrm{on }\He^1.
\end{equation}
From Lemma \ref{uper} we can restrict to study equation~\eqref{eq:HJ1} in $\Te$.
Following the procedure used for the unbounded case, see \cite[Section 3.2]{MMT} with $\varepsilon=0$, we can prove:
\begin{lemma}\label{valuefunction}
The value function $u$, defined in~\eqref{repr}, is the unique continuous bounded viscosity solution to problem~\eqref{eq:HJ1}. Moreover $u$ is $Q_{\cH}$-periodic, Lipschitz continuous w.r.t. $x$ and $t$, semiconcave w.r.t.~$x$ in~$Q_\cH$.
\end{lemma}
Now we want to state our optimal synthesis result. To this end, we need the notion of {\it $\cH$-differentiability} which is the differentiability following horizontal lines, extending the notion of Euclidean differentials (see~\cite[section 3.1]{CS}) to Heisenberg group; for the precise definition and main properties, we refer the reader to~\cite[Appendix A]{MMTa}.
\begin{lemma}\label{BB}
For $(x,t)\in\He^1\times (0,T)$, let $x(\cdot)$ be an AC function such that~$x(t)=x\in\He^1$
and for almost every $s\in (t,T)$,
$u(\cdot,s)$ is $\cH$-differentiable at $x(s)$.
Assume
\begin{equation}\label{OS}
x'(s)=-D_\cH u(x(s),s)B^T(x(s)), \qquad\textrm{a.e. }s\in(t,T).
\end{equation}
Then the control law $\alpha(s)$, given by 
$\alpha(s)=-D_\cH u(x(s),s)$ 
is optimal for $u(x,t)$.
\end{lemma}

\begin{proof}
We adapt the arguments of \cite[Lemma 3.6]{MMMT} and \cite[Lemma 4.11]{C}. 
Consider a function~$x(\cdot)$ as in the statement; note that this implies that $D_\cH u$ exists at $(x(s),s)$ for a.e. $s\in (t,T)$. 
We claim that $x(\cdot)$ is Lipschitz continuous. Indeed equation~\eqref{OS} reads
\begin{equation*}
x_i'(s)=-X_iu(x(s),s) \quad\textrm{for $i=1,2$},\quad x_3'(s)= x_2(s)X_1u(x(s),s)-x_1(s)X_2u(x(s),s)
\end{equation*}
By Lemma~\ref{valuefunction}, $D_\cH u$ is bounded; hence, $x_1(\cdot)$ and $x_2(\cdot)$ are Lipschitz continuous and, in particular, bounded. We infer that also $x_3(\cdot)$ is Lipschitz continuous. The claim is proved. 
The rest of the proof follows the arguments of \cite{MMMT,C} so we omit it.
\end{proof}

\subsection{The continuity equation}\label{sect:continuity}
Throughout this section we assume $(H1)$-$(H3)$ and we study equation~\eqref{eq:MFGintrin}-$(ii)$, namely
\begin{equation}\label{continuitye}
\partial_t m-\diver_{\cH}  (m D_{\cH}u)=0 \quad \textrm{in }\He^1\times (0,T),\quad m(x,0)=m_0(x) \quad \textrm{on }\He^1,
\end{equation}
where, for $\overline m$ fixed in $C^{1/4}([0,T],\mathcal P_{per}(\He^1))$, $u$ solves
\begin{equation}\label{HJe}
-\partial_t u+\frac{|D_{\cH}u|^2}{2}=F[\overline{m}_t](x) \quad \textrm{in }\He^1\times (0,T), \qquad u(x,T)=G[\overline {m}_T](x)\quad \textrm{on }\He^1.
\end{equation}
Let us observe that assumptions $(H1)$-$(H3)$ and Lemma~\ref{valuefunction} ensure that there is a unique bounded solution~$u$ to~\eqref{HJe} which is moreover~$Q_\cH$-periodic and Lipschitz continuous.
\begin{theorem}\label{prp:m}
Under assumptions (H1)-(H3), for any $\overline m\in C^{1/4}([0,T],\mathcal P_{per}(\He^1))$, problem~\eqref{continuitye} has a solution $m$ 
in the sense of Definition~\ref{defsolmfg}. Moreover the function $m$ belongs to $C^{1/4}([0,T],\mathcal P_{per}(\He^1))\cap {\mathbb L}^\infty(\He^1\times (0,T))$ and there exist two positive constants $C_0$ and $C_1$   (both independent of~$\overline m$) such that
\begin{equation}\label{stimaC0}
  0\leq m(x,t)\leq C_0 \qquad \forall (x,t)\in\He^1\times [0,T],
\end{equation}
\begin{equation}\label{stima0.5}
{\bf d_{1}}(m_s,m_t)\leq C_1(t-s)^{1/4} \qquad \forall\ 0\leq s\leq t\leq T. 
\end{equation}
\end{theorem}
\begin{proof}
Since the proof is inspired by the proof of~\cite[Theorem 4.1]{MMT}, here we only sketch the main differences.\\
\noindent 1.
We use a vanishing viscosity approach applied to the {\it whole} MFG system in terms of the {\it horizontal Laplacian}~$\Delta_\cH$. We need such approximation to ensure that the corresponding solution is still $Q_\cH$-periodic in~$x$.
More precisely, for any $\sigma \in(0,1)$, we consider the system
\begin{equation}
\label{eq:MFGv}
\left\{
\begin{array}{lll}
&(i)\quad-\partial_t u-\sigma \Delta_\cH u+\frac12 |D_\cH u|^2
=F[\overline {m}_t](x)&\qquad \textrm{in }\He^1\times (0,T),\\
&(ii)\quad \partial_t m-\sigma\Delta_\cH m-\diver_\cH (m D_\cH u)=0&\qquad \textrm{in }\He^1\times (0,T),\\
&(iii)\quad m(x,0)=m_0(x), u(x,T)=G[{\overline {m}}_T](x)&\qquad \textrm{on }\He^1.
\end{array}\right.
\end{equation}
\noindent 2. We use the following two lemmas whose proofs are postponed after proof of the theorem.
\begin{lemma}\label{lm:usigma}
There is a unique bounded solution $u^\sigma$ to problem~\eqref{eq:MFGv}-$(i)$,-$(iii)$. Moreover,
\begin{itemize}
\item[(i)] $u^\sigma$ is $Q_\cH$-periodic and there exists $C>0$ (independent of $\sigma$ and of $\overline m$) such that:\\
$u^\sigma$ is semiconcave in~$x$ with semiconcavity constant~$C$, 
\[
\|u^\sigma\|_{L^\infty(\He^1\times[0,T])}\leq C,\, \|D_\cH u^\sigma\|_{L^\infty(\He^1\times[0,T])}\leq C,\, \Delta_\cH u^\sigma(x,t) \leq C, \,\forall(x,t)\in \He^1\times[0,T],
\]
\item[(ii)] for every $\tau\in[0,T)$ and $\delta\in(0,1/4]$, there exists a positive constant $\bar C$ (depending on $\tau$, $\delta$ and $\sigma$) such that
\[
\|u^\sigma\|_{C^{2+\delta}_{\cH}(\He^1\times[0,\tau])}\leq \bar C,
\]
\item[(iii)] the functions~$u^\sigma$ are $1/4$-H\"older continuous in time uniformly in~$\sigma$.
\end{itemize}
\end{lemma}

\begin{lemma}\label{lm:msigma}
Problem~\eqref{eq:MFGv}-$(ii)$,-$(iii)$ admits exactly one bounded classical solution~$m^\sigma$ in $C^0(\He^1\times[0,T])$. Moreover, $m^\sigma$ has the following properties:
\begin{itemize}
\item[(i)]  $m^\sigma$ is $Q_\cH$-periodic in~$x$ and there exists a constant~$C_0>0$ (independent of~$\sigma$ and of $\overline m$) such that
\begin{equation*}
  0\leq m^\sigma(x,t)\leq C_0 \qquad \forall (x,t)\in\He^1\times [0,T],
\end{equation*}
\item[(ii)] for every $0<t_1< t_2<T$ and $\delta\in(0,1/4]$, there exists $C_1>0$ (depending on $\sigma$, $t_1$, $t_2$ and $\delta$) such that $\|m^\sigma\|_{C^{2+\delta}_\cH(\He^1\times[t_1, t_2])}\leq C_1$.
\end{itemize}
\end{lemma}
\noindent 3. As for the Euclidean case (for instance, see \cite[Lemma 3.4]{C}) it is expedient to interpret $m^\sigma$ as the law of a suitable stochastic process. In fact, we shall adapt this approach for the present setting where $m_0$ is only a nonnegative measure on $\He^1$ and the coefficients in the SDE are unbounded. To this end, we consider a probability space $(\Omega, {\mathcal F}, P)$, equipped with a filtration~$({\mathcal F}_t)_{t\geq 0}$ and a standard $2$-dimensional $({\mathcal F}_t)$-adapted Wiener process~$W_{\cdot}$. For any $x\in \He^1$, we introduce the process
\begin{equation}\label{processMarkus}
d Y^x_t= -D_\cH u^\sigma(Y^x_t,t) B^T(Y^x_t) dt +\sqrt{2\sigma} B(Y^x_t) d W_t,\qquad Y^x_0=x.
\end{equation}
By Lemma~\ref{lm:usigma}, the coefficients in~\eqref{processMarkus} are locally Lipschitz continuous for $t\in[0,T)$ with an at most linear growth as $Y\to\infty$; hence, by standard theory  (for instance, \cite[Theorem 8.10]{Baldi} or also \cite[Theorem 9.2]{Baldi2} and \cite [theorem B.3.1]{BGL}) there exists a unique solution to~\eqref{processMarkus} in $[0,T)$.
We set
\begin{equation}\label{etaMarkus}
\eta_t^\sigma:=\int_{\He^1}{\mathcal L}(Y^x_t)dm_0(x), \quad  t\in [0,T), 
\end{equation}
where ${\mathcal L}(Y^x_t)$ is the law of the process $Y^x_t$ and we prove that it coincides with $m_t^\sigma$ in $[0,T)$:

\begin{lemma}\label{lm:etaMarkus}
The function $\eta_{\cdot}^\sigma:[0,T)\rightarrow {\mathcal M}(\He^1)$ is $Q_\cH$-periodic (in the sense of~\eqref{eq:Pper}) with $\eta_t^\sigma(Q_\cH)=1$ for every $t\in[0,T)$. Moreover, it fulfills: for a constant~$C_1>0$ (independent of $\sigma$ and~$\overline m$),  
\begin{equation}\label{etaMarkus3}
{\bf d_{1}}(\eta_t^\sigma, \eta_s^\sigma)\leq C_1(t-s)^{1/4} \qquad \forall 0\leq s\leq t<T
\end{equation}
and, for every $\phi\in C^{2,1}(\He^1\times[0,T])$ with support in a compact of~$\He^1\times [0,T]$ and $t\in(0,T)$,
\begin{equation}\label{eq:FPv_distr}
\int_{\He^1}\phi(x,t)\eta^\sigma_t(dx)=\int_{\He^1}\phi(x,0)m_0(x)dx+
\iint_{[0,t]\times\He^1} [\partial_t \phi +\sigma \Delta_\cH\phi-D_\cH u^\sigma\cdot D_\cH \phi]\eta^\sigma_s(dx)ds;
\end{equation}
in particular, it coincides with $m_t^\sigma$ in $[0,T)$.
\end{lemma}

\noindent 4.
We follow the arguments of the proof of \cite[Theorem 5.1]{C13} (see also \cite[Theorem 4.20]{C}).
By the estimates in Lemma~\ref{lm:usigma}-($ii$) and ($iii$), possibly passing to a subsequence (that we still denote by~$u^\sigma$), as $\sigma\to 0^+$, the sequence~$\{u^\sigma\}_\sigma$ uniformly converges to the function~$u$ which solves~\eqref{HJe}, is $1/4$-H\"older continuous in time and horizontally Lipschitz continuous in space, with $D_\cH u^\sigma \to D_\cH u$ a.e. (by Lemma~\ref{lm:usigma}-$(i)$ and \cite[Theorem 3.3.3]{CS}).

On the other hand, 
we note that, Lemma~\ref{lm:etaMarkus} permits to identify $m^\sigma(\cdot,t)$ with the density of a measure in $\mathcal P_{per}(\He^1)$, namely with the density of a measure on the compact set~$\Te$; moreover, by continuity (established in Lemma~\ref{lm:msigma}), the function $m^\sigma$ fulfills~\eqref{etaMarkus3} on the whole interval~$[0,T]$.
The estimates for $m^\sigma$ in Lemma~\ref{lm:msigma}-$(i)$ and in~\eqref{etaMarkus3} ensure that, as $\sigma\to 0^+$, possibly passing to a subsequence, $\{m^\sigma\}_\sigma$ converges to some $m\in C^{1/4}([0,T],\mathcal P_{per}(\He^1))$ in the $C^{0}([0,T],\mathcal P_{per}(\He^1))$-topology and in the $\mathbb L^\infty(\He^1\times(0,T))$-weak-$*$ topology; $m$ satisfies \eqref{stimaC0} with the same constant $C_0$ of Lemma~\ref{lm:msigma}-$(i)$
and \eqref{stima0.5} with the same constant $C_1$ of ~\eqref{etaMarkus3}. 
In conclusion, we accomplish the proof arguing as in \cite[Proposition 3.1(proof)]{MMMT}.
\end{proof}

We now give the proofs of the Lemmas \ref{lm:usigma}, \ref{lm:msigma}, \ref{lm:etaMarkus}.

\begin{Proofc}{Proof of Lemma \ref{lm:usigma}.}
Existence, uniqueness, semiconcavity and the first two estimates in point~$(i)$ can be proved by the same arguments of \cite[Lemma 4.1]{MMT}. The last estimate in point~$(i)$ easily follows from semiconcavity and periodicity.\\
$(ii)$. Fix $\tau$ and $\delta$ as in the statement. The Cole-Hopf transformation of $u^\sigma$, $w^\sigma(x,t):=\exp\{-u^\sigma(x,t)/(2\sigma)\}$ is bounded, $Q_\cH$-periodic and is a viscosity solution to
\begin{equation}\label{CHtrans}
\left\{
\begin{array}{lll}
&-\partial_t w^\sigma -\sigma \Delta_\cH w^\sigma + w^\sigma F[\overline {m}]/(2\sigma)=0&\qquad \textrm{in }\He^1\times (0,T),\\
&w^\sigma(x,T)=\exp\{-G[{\overline {m}}_T](x)/(2\sigma)\}&\qquad \textrm{on }\He^1;
\end{array}\right.
\end{equation}
by the equivalence between distributional solutions and viscosity solutions (established in~\cite{I95} for the elliptic case but holding also in the evolutive one), $w^\sigma$ is also a distributional solution of equation~\eqref{CHtrans}.\\
Consider a bounded domain $Q'\subset \He^1$ with $\overline{Q_\cH}\subset Q'$. Since $F[\overline {m}]$ belongs to $C^{1/4}_{\cH}(\He^1\times[0,T])$, classical results for linear subelliptic operators,~\cite[Theorem 10.7]{BBLU} and~\cite[Theorem 1.1]{BB07} ensure that, for every $0\leq t_1<t_2<T$, the function~$w^\sigma$ belongs to $C^{2+\delta}_{\cH}(Q_\cH\times[0,t_1])$ and there exists a constant $C'$ (depending on~$t_1$,$t_2$, $\sigma$ and $\delta$) such that
\begin{equation}\label{primaiii}
\|w^\sigma\|_{C^{2+\delta}_{\cH}(Q_\cH\times[0,t_1])}\leq C' \|w^\sigma\|_{L^\infty(Q'\times[0,t_2])}.
\end{equation}
Choosing $t_1=\tau$ and $t_2=(T+\tau)/2$, by periodicity and the first estimate in point~$(i)$, we accomplish the proof.

\noindent $(iii)$. We follow the arguments of~\cite[Lemma 3.4]{MMMT} and~\cite[Theorem 5.1]{C13} so we only provide their main steps. For some~$C_1>0$ (independent of~$\sigma$), the function $w^+(x,t):=G[{\overline {m}}_T](x)+ C_1(T-t)$ is a supersolution to~\eqref{eq:MFGv}-$(i)$. The standard comparison principle yields $u^\sigma(x,t)\leq G[{\overline {m}}_T](x)+ C_1(T-t)$.
Moreover, since $\|F[{\overline {m}}_t]-F[{\overline {m}}_{(t-h)}]\|_\infty\leq C_2h^{1/4}$, the function $v^\sigma_h(x,t):=u^\sigma(x,t-h)+C_1 h+C_2h^{1/4}(T-t)$ is a supersolution to \eqref{eq:MFGv}-$(i)$ with $v^\sigma_h(x,T)\geq u^\sigma(x,T)$. The comparison principle entails $u^\sigma(x,t-h)- u^\sigma(x,t)\geq-C_1 h-C_2h^{1/4}(T-t)$. The other inequality is obtained similarly.
\end{Proofc}
\begin{remark}
Under the assumptions of Lemma \ref{lm:usigma}-$(ii)$, there exists a positive constant $C$ (depending on $\tau$, $\delta$ and $\sigma$) such that
\[
\sum_{i=1}^2\|X_iu^\sigma\|_{C^{2+\delta}_{\cH}(\He^1\times[0,\tau])}+
\sum_{i,j=1}^2\|X_iX_j u^\sigma\|_{C^{2+\delta}_{\cH}(\He^1\times[0,\tau])}\leq C.
\]
Note that this estimate holds ``away'' from time~$T$; hence, the regularity of the datum~$G$ plays no role. In order to prove it, we proceed with a bootstrap of~\eqref{primaiii}.  We first remark that $W_3:=\partial_{x_3}w^\sigma=(X_1X_2 w^\sigma-X_2X_1 w^\sigma)/2=[X_1,X_2]w^\sigma/2$ is a distributional solution to the pde in~\eqref{CHtrans} with $-w^\sigma \partial_{x_3}F[\overline {m}]/(2\sigma)$ in the right-hand side; hence, again by \cite[Theorem 10.7]{BBLU} and~\cite[Theorem 1.1]{BB07}, $\partial_{x_3}w^\sigma\in C^{2+\delta}_{\cH}(Q_{\cH}\times[0,t_1])$ and there exists a constant $C'$ (depending on~$t_1$,$t_2$, $\sigma$ and $\delta$) such that
\begin{equation*}
\|\partial_{x_3}w^\sigma\|_{C^{2+\delta}_{\cH}(Q_{\cH}\times[0,t_1])}\leq C'\left(\|w^\sigma \partial_{x_3}F[\overline {m}]/(2\sigma)\|_{C^{\delta}_{\cH}(Q'\times[0,t_2])}+\|\partial_{x_3}w^\sigma\|_{L^\infty(Q'\times[0,t_2])}\right).
\end{equation*}
By periodicity and Lemma~\ref{lm:usigma}-($ii$), we deduce that there exists a constant $C$ such that
\begin{equation}\label{partial3}
\|\partial_{x_3}w^\sigma\|_{C^{2+\delta}_{\cH}(\He^1\times[0,t_1])}\leq C.
\end{equation}
Then the function $W_1:=X_1 w^\sigma$ is a distributional solution to the pde in~\eqref{CHtrans} with $4\sigma\, X_2\partial_{x_3}w^\sigma
-w^\sigma X_1F[\overline {m}]/(2\sigma)$ in the right-hand side.
Repeating the same arguments as before and by estimate~\eqref{partial3} (where we increase the value $t_1$), we get 
$\|X_1 w^\sigma\|_{C^{2+\delta}_{\cH}(\He^1\times[0,t_1])}\leq C$. By similar arguments, we also get $\|X_2 w^\sigma\|_{C^{2+\delta}_{\cH}(\He^1\times[0,t_1])}\leq C$.\\
Repeating these arguments. we obtain the bound for $X_iX_jw^\sigma$. Reversing the Cole-Hopf transformation, we obtain the desired estimates.
\end{remark}

\begin{Proofc}{Proof of Lemma \ref{lm:msigma}.}
 Lemma~\ref{lm:usigma}-$(ii)$ ensures that the drift $D_\cH u$ belongs to $C^{1+\delta}_\cH(\He^1\times[0,\tau))$ for any~$\delta\in(0,1/4]$ and~$\tau\in(0,T)$; hence, the equation~\eqref{eq:MFGv}-$(ii)$ can be written as $\partial_t m-\sigma\Delta_\cH m-D_\cH m\cdot D_\cH u - m\Delta_\cH u=0$. By standard parabolic theory (see \cite{Lie}, \cite[Theorem 10.7]{BBLU} and also \cite[Lemma 4.1]{MMT}), we get existence and uniqueness of a bounded classical solution~$m^\sigma$ to~\eqref{eq:MFGv}-$(ii)$ with $m^\sigma\in C^{0}(\He^1\times[0,T])$.\\
$(i)$. The $Q_\cH$-periodicity follows from standard arguments and uniqueness of the solution.
By~\cite[Proposition 3.1]{MMM} (see also \cite[Lemma 4.2]{MMT}), we get $0\leq m^\sigma\leq C_0$.\\
$(ii)$. By the periodicity, \cite[Theorem 10.7]{BBLU} and \cite[Theorem 1.1]{BB07}, we accomplish the proof.
\end{Proofc}

\begin{Proofc}{Proof of Lemma \ref{lm:etaMarkus}}
By standard theory, one obtains the translation formula $z\oplus Y^x_t=Y^{z\oplus x}_t$ for every $z\in\Z^3$, $x\in\He^1$, $t\in [0,T)$ and, consequently, that~$\eta^\sigma$ is $Q_\cH$-periodic. 
By the property of pavage and the periodicity of~$m_0$, one can also obtain $\eta^\sigma_t(Q_\cH)=1$.

We now prove~\eqref{etaMarkus3}. Using Remark \ref{rmk:misureper}, we denote by $\eta_t^\sigma$ also the corresponding probability measure on $Q_{\cH}$ and recall that $q_\cH(\cdot)$ is the projection introduced in section~\ref{sub:periodicity} while $\Pi$ is the set introduced in~\eqref{Pi}.
Fix $0\leq s\leq t<T$; for each $x\in\He^1$, introduce
\begin{equation*}
\tilde \pi:=\int_{\Te} {\mathcal L}(Y^{per,x}_s,Y^{per,x}_t)dm_0(x)
\end{equation*}
where $Y^{per,x}_\tau:= q_\cH(Y^{x}_\tau)$  and ${\mathcal L}(Y^{per,x}_s,Y^{per,x}_t)$ is the law of $(Y^{per,x}_s,Y^{per,x}_t)$.
Again by the translation formula $z\oplus Y^x_t=Y^{z\oplus x}_t$, the property of pavage and the periodicity of~$m_0$, one obtains: $\tilde \pi\in \Pi(\eta_{s}^\sigma,\eta_{t}^\sigma)$.
Then, there holds
\begin{eqnarray*}
{\bf d_{1}}(\eta_t^\sigma, \eta_s^\sigma)&\leq&\int_{\Te\times \Te}d_{\Te}(z_1,z_2)\tilde \pi(dz_1,dz_2)=\int_{\Te}\mathbb E[d_{\Te}(Y^{per,x}_s,Y^{per,x}_t)]dm_0(x)
\\&\leq&\int_{\Te}\mathbb E\left[|Y^{per,x}_s-Y^{per,x}_t|^{1/2}\left(1+2|Y^{per,x}_s|^{1/2}+|Y^{per,x}_s-Y^{per,x}_t|^{1/2}\right)\right]dm_0(x).
\end{eqnarray*}
where the last inequality is due to $d_\cH(x,y)\leq |x-y|+(1+|x_1|^{1/2}+|x_2|^{1/2})|x-y|^{1/2}.$
Since now on we denote by~$C$ a constant which may change from line to line but which is independent of $x,s,t,\sigma,\overline m$. Since $|Y^{per,x}_s|^{1/2},|Y^{per,x}_s-Y^{per,x}_t|^{1/2}\leq \sqrt3$,  we get
\begin{equation} \label{stimad1}
{\bf d_{1}}(\eta_t^\sigma, \eta_s^\sigma) \leq C\int_{\Te}\mathbb E\left[\left(\int_s^t |D_\cH u^\sigma B^T| d\tau\right)^{1/2}  +\sigma^{1/4}\left|\int_s^t B d W_\tau\right|^{1/2}\right]dm_0(x).
\end{equation}
We claim that there exists a constant~$C$, independent of~$\sigma$, such that
\begin{equation}\label{stimaL2Markus}
\mathbb E[|Y^{x}_\tau|^2]\leq C\qquad\forall x\in \Te, \quad 0\leq \tau<T.
\end{equation}
Indeed, we have $\mathbb E[|Y^{x}_0|^2]=|x|^2$ for every $x\in \Te$ and, by Lemma~\ref{lm:usigma}, the coefficients in the SDE~\eqref{processMarkus} are locally Lipschitz continuous and have an at most linear growth as $Y\to\infty$ since they are bounded by $4\|D_\cH u^\sigma\|_{\infty}(1+|Y|)$ (which, in turns, is bounded independently of~$\sigma$ by Lemma~\ref{lm:usigma}-$(i)$). By standard theory (see \cite[Theorem 8.10]{Baldi} or also \cite[Theorem 9.4]{Baldi2}) we deduce~\eqref{stimaL2Markus} with a constant $C$ which depends on the constant of Lemma~\ref{lm:usigma}-$(i)$ but is independent of~$\sigma$ and of the Lipschitz constant of the coefficients. Hence, our claim~\eqref{stimaL2Markus} is proved.\\
By Jensen inequality and by Fubini theorem, there holds
\begin{eqnarray*}
\int_{\Te}\mathbb E\left[\left(\int_s^t |D_\cH u^\sigma B^T| d\tau\right)^{1/2}\right]dm_0(x) &\leq&
\int_{\Te}\left(\int_s^t \mathbb E [ |D_\cH u^\sigma (Y^{x}_\tau,\tau)B^T(Y^{x}_\tau)|] d\tau\right)^{1/2}dm_0(x) \\
&\leq&C \int_{\Te}\left(\int_s^t\mathbb E[1+| Y^{x}_\tau|^2]d\tau \right)^{1/2}dm_0(x)
\end{eqnarray*}
where the last inequality is due to Lemma~\ref{lm:usigma}, the definition of~$B$ in~\eqref{matrixB} and the Cauchy-Schwarz inequality.
Using estimate~\eqref{stimaL2Markus} in the previous inequality, we obtain
\begin{equation}\label{stimad1A}
\int_{\Te}\mathbb E\left[\left(\int_s^t |D_\cH u^\sigma B^T| d\tau\right)^{1/2}\right]dm_0(x) \leq C\sqrt{t-s}.
\end{equation}

On the other hand, by Jensen inequality and Cauchy-Schwarz inequality, we get
\begin{equation*}
\begin{array}{l}
\int_{\Te}\mathbb E\left[\left|\int_s^t B d W_\tau\right|^{1/2}\right]dm_0(x)\leq
\int_{\Te}\left(\mathbb E\left[\left|\int_s^t B(Y^{x}_\tau) d W_\tau\right|\right] \right)^{1/2}dm_0(x)\\
\qquad\leq\int_{\Te}\left(\mathbb E \left[\left|\int_s^t B(Y^{x}_\tau) d W_\tau\right|^2\right] \right)^{1/4}dm_0(x)
\leq\int_{\Te}\left(\mathbb E\left[\int_s^t(1+|Y^{x}_\tau|^2)d\tau \right] \right)^{1/4}dm_0(x)
\end{array}
\end{equation*}
where the last inequality is due to standard calculus for Ito's integral.
Using again Fubini theorem and estimate~\eqref{stimaL2Markus} in the previous inequality, we get
\begin{equation}\label{stimad1B}
\int_{\Te}\mathbb E\left[\left|\int_s^t B d W_\tau\right|^{1/2}\right]dm_0(x)\leq C(t-s)^{1/4}.
\end{equation}
Replacing estimates~\eqref{stimad1A} and \eqref{stimad1B} in~\eqref{stimad1}, taking $\sigma \in(0,1)$, we obtain estimate~\eqref{etaMarkus3}.\\
Moreover, equation~\eqref{eq:FPv_distr} is due to a standard application of Ito's formula as in the Euclidean setting (for instance, see~\cite[Lemma 3.3]{C} and also~\cite[Theorem 5.7.6]{KS}). Finally, by uniqueness of distributional solution to~\eqref{eq:MFGv}-$(ii)$,-$(iii)$ (see \cite[Proposition B.2]{MMTa}), we achieve that $\eta^\sigma$ coincides with $m^\sigma$.
\end{Proofc}

\subsection{Proof of Theorem \ref{thm:main}}\label{sect:dim3.1}

\begin{Proofc}{Proof of Theorem \ref{thm:main}}
$(i)$ The set
\[
{\cal C} :=\left\{m\in C^{1/4}([0,T]; {{\cal P}_{per}(\He^1)}): \textrm{$m$ fulfills \eqref{stimaC0}-\eqref{stima0.5} and $m(0)=m_0$} \right\},
\]
endowed with the $C^0([0,T]; {{\cal P}_{per}(\He^1)})$-topology, is a nonempty, convex, compact subset of~$C^0([0,T]; {{\cal P}_{per}(\He^1)})$.
We introduce the following set valued map~${\cal T}$ on~${\cal C}$: for any $\overline m\in {\cal C}$, 
\[{\cal T}(\overline m):=\left\{m\in C^{1/4}([0,T]; {\cal P}_{per}(\He^1)):\begin{array}{l}\textrm{$m$ solves \eqref{continuitye} and fulfills \eqref{stimaC0}-\eqref{stima0.5}} \end{array}\right\}.
\]  
Clearly it is enough to prove that~${\cal T}$ has a fixed point. To this end, we apply Kakutani's Theorem; in fact, here we cannot use Schauder's theorem as in \cite[Theorem 4.1]{C} because we do not have uniqueness of the solution to \eqref{continuitye}.
We note that Theorem~\ref{prp:m}  and the linearity of~\eqref{continuitye} ensure $\emptyset\ne{\cal T}(\overline m)\subseteq {\cal C}$ and ${\cal T}(\overline m)$ is convex.
We claim that ${\cal T}$ has closed graph. Indeed, let us consider $\overline m_n, \overline m\in {\cal C}$ and 
$m_n\in {\cal T}(\overline m_n)$ with $\overline m_n\rightarrow \overline m$ and $m_n\rightarrow m$ in the $C^0([0,T]; {{\cal P}_{per}(\He^1)})$-topology.
By assumptions $(H1)$-$(H2)$, possibly passing to a subsequence (that we still denote $\overline m_n$),  Ascoli-Arzela theorem guarantees that $F[\overline m_n]$ and $G[\overline m_n(T)]$ converge uniformly to $F[\overline m]$ in  $\Te\times [0,T]$ and, respectively, to $G[\overline m(T)]$ in $\Te$. Moreover, Lemma~\ref{valuefunction} ensures that the solutions $u_n$ to problem~\eqref{HJe} with $\overline m$ replaced by $\overline m_n$ are $Q_\cH$-periodic, uniformly bounded and uniformly Lipschitz continuous. By standard stability results, the functions~$u_n$ converge uniformly to the solution~$u$ to~\eqref{HJe}.
Moreover, still by Lemma \ref{valuefunction}, the functions $u_n$ are uniformly semiconcave; hence by \cite[Theorem 3.3.3]{CS} $Du_n$ converges a.e. to $Du$.
On the other hand, by definition of ${\cal T}$, the functions $m_n\in {\cal T}(\overline m_n)$ are uniformly bounded and uniformly $1/4$-H\"older continuous, so by Ascoli-Arzela theorem and Banach-Alaoglu theorem, there exists a subsequence $\{m_{n_k}\}_k$ which converges to $m$ in the $C^0([0,T]; {{\cal P}_{per}(\He^1)})$-topology and in the $\mathbb L^\infty(\Te\times[0,T])$-weak-$*$ topology. Solving \eqref{continuitye}-\eqref{HJe} with $\overline m$ replaced by $\overline m_{n_k}$, the function $m_{n_k}$ fulfills
\begin{equation*}
\int_0^T\int_{\mathbb{H}^{1}}m_{n_k}(-\partial_t\varphi+D_{\cH}u_{n_k}\cdot D_{\cH}\varphi)dxdt=0\qquad \forall \varphi\in C_{c}^{\infty}(\mathbb{H}^{1}\times (0,T)).
\end{equation*}
Passing to the limit as $k\rightarrow +\infty$ we get that $m$ is a solution to \eqref{continuitye}. Moreover again by the uniform convergence and the uniform $1/4$-H\"older continuity of $m_{n_k}$, $m$ satisfies \eqref{stimaC0}-\eqref{stima0.5}. In conclusion $m\in {\cal T}(\overline m)$ and our claim is proved.
Then, Kakutani's Theorem guarantees the existence of a fixed point for~${\cal T}$.

$(ii)$ Consider the function $m$ found in point $(i)$. Since $t\rightarrow m_t$ is narrowly continuous, applying Theorem \ref{821}, we get that there exists a probability measure $\eta^*$ in $\Te\times\Gamma$ which satisfies points $(i)$ and $(ii)$ of Theorem \ref{821}.
We denote $\eta\in {\mathcal P}(\Gamma)$ the measure on $\Gamma$ defined as $\eta(A):=\eta^*(\Te\times A)$ for every $A\subset \Gamma$ measurable. We claim that $\eta$ is a MFG equilibrium. Indeed, by \eqref{821nostra}, we have $e_0\#\eta=m_0$ and $e_t\#\eta\in\mathcal P_{per}(\He^1)$,
so $\eta\in {\mathcal P}_{m_0}(\Gamma)$. 
Moreover from $(i)$ of Theorem \ref{821}, $\eta$ is supported on the curves solving \eqref{ODE}. From Lemma \ref{BB} such curves are optimal, i.e. they belong to the set $\Gamma^{\eta}[x]$. Hence $\eta$ is a MFG equilibrium, our claim is proved.\\
Let us now prove that $(u,m)$ is a mild solution. By \eqref{821nostra}, we have $m_t=e_t\#\eta$. Moreover, by Lemma \ref{valuefunction}, the function $u$ found in point~$(i)$ is the value function associated with $m$  as in Definition \ref{mild}-$(ii)$. In conclusion $(u,m)$ is a mild solution to \eqref{eq:MFGintrin}.
\end{Proofc}
\begin{remark}
An alternative proof of Theorem \ref{thm:main}-$(i)$ could be as follow. First one obtains the existence of a solution to system
\begin{equation}
\label{eq:MFGvv}
\left\{
\begin{array}{lll}
&(i)\quad-\partial_t u^\sigma-\sigma \Delta_\cH u^\sigma+\frac12 |D_\cH u^\sigma|^2
=F[m^\sigma_t](x)&\qquad \textrm{in }\He^1\times (0,T),\\
&(ii)\quad \partial_t m^\sigma-\sigma\Delta_\cH m^\sigma-\diver_\cH (m^\sigma D_\cH u^\sigma)=0&\qquad \textrm{in }\He^1\times (0,T),\\
&(iii)\quad m^\sigma(x,0)=m_0(x), u^\sigma(x,T)=G[m^\sigma_T](x)&\qquad \textrm{on }\He^1.
\end{array}\right.
\end{equation}
applying Schauder fixed point theorem to the map 
 $\overline{\cal T}:{\cal C}\rightarrow {\cal C}$, $\overline{\cal T}(\overline m)=m$ where $m$ solves~\eqref{eq:MFGv}-$(ii)$ (and $u$ solves~\eqref{eq:MFGv}-$(i)$). After, as $\sigma\to 0$, by arguments similar to the above ones, one gets a solution to \eqref{eq:MFGintrin}.
\end{remark}

\noindent{\bf Acknowledgments.} 
The authors are grateful to the anonymous referees for their fruitful comments
and suggestions.
The first author is member of the Gnampa-INDAM group and she was partially supported by the MUR Excellence Department Project 2023-2027 MatMod@Tov CUP:E83C23000330006 awarded to the Department of Mathematics, University of Rome Tor Vergata.
The second and the third authors are members of GNAMPA-INdAM and were partially supported by the research project of the University of Padova ``Mean-Field Games and Nonlinear PDEs'' and by the Fondazione CaRiPaRo Project ``Nonlinear Partial Differential Equations: Asymptotic Problems and Mean-Field Games''. The forth author has been partially funded by the ANR project ANR-16-CE40-0015-01.

{\small{
 }

\end{document}